\documentclass[12pt]{amsart}
\usepackage{amsmath,amsthm,amsfonts,amssymb,mathrsfs}
\date{\today}

\usepackage{color}

\usepackage{hyperref}

\newtheorem{theorem}{Theorem}

\newtheorem{proposition}[theorem]{Proposition}
\newtheorem{corollary}[theorem]{Corollary}
\newtheorem{lemma}[theorem]{Lemma}
\theoremstyle{definition}
\newtheorem{example}[theorem]{Example}
\newtheorem{remark}[theorem]{Remark}

\newtheorem{definition}[theorem]{Definition}

\begin{document}

\title[On a complete topological inverse polycyclic monoid]{On a complete topological inverse polycyclic monoid}

\author{Serhii Bardyla}

\author{Oleg Gutik}
\address{Faculty of Mathematics, National University of Lviv,
Universytetska 1, Lviv, 79000, Ukraine}
\email{sbardyla@yahoo.com}
\email{o\underline{\hskip5pt}\,gutik@franko.lviv.ua,
ovgutik@yahoo.com}



\begin{abstract}
We give sufficient conditions when a topological inverse $\lambda$-polycyclic monoid $P_{\lambda}$ is absolutely $H$-closed in the class of topological inverse semigroups. Also, for every infinite cardinal $\lambda$ we construct the coarsest semigroup inverse topology $\tau_{mi}$ on $P_\lambda$ and give an example of a topological inverse monoid S which contains the polycyclic monoid $P_2$ as a dense discrete subsemigroup.
\end{abstract}

\maketitle


In this paper all topological spaces will be assumed to be Hausdorff. We shall follow the terminology of~\cite{Carruth-Hildebrant-Koch-1983-1986, Clifford-Preston-1961-1967, Engelking-1989, Lawson-1998}. If $A$ is a subset of a topological space $X$, then we denote the closure of the set $A$ in $X$ by
$\operatorname{cl}_X(A)$. By $\mathbb{N}$ we denote the set of all positive integers and by $\omega$ the first infinite cardinal.

A semigroup $S$ is called an \emph{inverse semigroup} if every $a$
in $S$ possesses an unique inverse, i.e. if there exists an unique
element $a^{-1}$ in $S$ such that
\begin{equation*}
    aa^{-1}a=a \qquad \mbox{and} \qquad a^{-1}aa^{-1}=a^{-1}.
\end{equation*}
A map which associates to any element of an inverse semigroup its
inverse is called the \emph{inversion}.

A \emph{band} is a semigroup of idempotents. If $S$ is a semigroup, then we shall denote the subset of idempotents in $S$ by $E(S)$. If $S$ is an inverse semigroup, then $E(S)$ is closed under multiplication. The semigroup operation on $S$
determines the following partial order $\leqslant$ on $E(S)$: $e\leqslant f$ if and only if $ef=fe=e$. This order is called the {\em natural partial order} on $E(S)$. A \emph{semilattice} is a commutative semigroup of idempotents. A semilattice $E$ is called {\em linearly ordered} or a \emph{chain} if its natural order is a linear order. A \emph{maximal chain} of a semilattice $E$ is a chain which is properly contained in no other chain of $E$. The Axiom of Choice implies the existence of maximal chains in any partially ordered set. According to \cite[Definition~II.5.12]{Petrich-1984} a chain $L$ is called $\omega$-chain if $L$ is order isomorphic to $\{0,-1,-2,-3,\ldots\}$ with the usual order $\leqslant$. Let $E$ be a semilattice and $e\in E$. We denote ${\downarrow} e=\{ f\in E\mid f\leqslant e\}$
and ${\uparrow} e=\{ f\in E\mid e\leqslant f\}$.

If $S$ is a semigroup, then we shall denote by $\mathscr{R}$,
$\mathscr{L}$, $\mathscr{D}$ and $\mathscr{H}$ the
Green relations on $S$ (see \cite{GreenJ1951} or \cite[Section~2.1]{Clifford-Preston-1961-1967}):
\begin{center}
\begin{tabular}{rcl}
  $a\mathscr{R}b$ & if and only if & $aS^1=bS^1$; \\
  $a\mathscr{L}b$ & if and only if & $S^1a=S^1b$; \\
     & $\mathscr{D}=\mathscr{L}{\circ}\mathscr{R}=\mathscr{R}{\circ}\mathscr{L}$; &\\
    & $\mathscr{H}=\mathscr{L}\cap\mathscr{R}$. &\\
\end{tabular}
\end{center}
The $\mathscr{R}$-class (resp., $\mathscr{L}$-, $\mathscr{H}$-, or $\mathscr{D}$--class) of the semigroup $S$ which contains an element $a$ of $S$ will be denoted by $R_a$ (resp., $L_a$, $H_a$, or $D_a$).

The bicyclic monoid ${\mathscr{C}}(p,q)$ is the semigroup with the identity $1$ generated by two elements $p$ and $q$ subjected only to the condition $pq=1$. The semigroup operation on ${\mathscr{C}}(p,q)$ is determined as
follows:
\begin{equation*}
    q^kp^l\cdot q^mp^n=q^{k+m-\min\{l,m\}}p^{l+n-\min\{l,m\}}.
\end{equation*}
It is well known that the bicyclic monoid ${\mathscr{C}}(p,q)$ is a bisimple (and hence simple) combinatorial $E$-unitary inverse semigroup and every non-trivial congruence on ${\mathscr{C}}(p,q)$ is a group congruence \cite{Clifford-Preston-1961-1967}. Also the well
known  Andersen Theorem states that \emph{a simple semigroup $S$ with an idempotent is completely simple if and only if $S$ does not contains an isomorphic copy of the bicyclic semigroup} (see \cite{Andersen-1952} and \cite[Theorem~2.54]{Clifford-Preston-1961-1967}).

Let $\lambda$ be a non-zero cardinal. On the set
 $
 B_{\lambda}=(\lambda\times\lambda)\cup\{ 0\}
 $,
where $0\notin\lambda\times\lambda$, we define the semigroup
operation ``$\, \cdot\, $'' as follows
\begin{equation*}
(a, b)\cdot(c, d)=
\left\{
  \begin{array}{cl}
    (a, d), & \hbox{ if~ } b=c;\\
    0, & \hbox{ if~ } b\neq c,
  \end{array}
\right.
\end{equation*}
and $(a, b)\cdot 0=0\cdot(a, b)=0\cdot 0=0$ for $a,b,c,d\in\lambda$. The semigroup $B_{\lambda}$ is called the \emph{semigroup of $\lambda{\times}\lambda$-matrix units}~(see \cite{Clifford-Preston-1961-1967}).

In 1970 Nivat and Perrot proposed the following generalization of the bicyclic monoid (see \cite{Nivat-Perrot-1970} and \cite[Section~9.3]{Lawson-1998}). For a non-zero cardinal $\lambda$, the polycyclic monoid on $\lambda$ generators $P_\lambda$ is the semigroup with zero given by the presentation:
\begin{equation*}
    P_\lambda=\big\langle \{p_i\}_{i\in\lambda}, \{p_i^{-1}\}_{i\in\lambda}\mid p_i p_i^{-1}=1, p_ip_j^{-1}=0 \hbox{~for~} i\neq j\big\rangle.
\end{equation*}
It is obvious that in the case when $\lambda=1$ the semigroup $P_1$ is isomorphic to the bicyclic semigroup with adjoined zero. For every finite non-zero cardinal $\lambda=n$ the polycyclic monoid $P_n$ is a congruence free, combinatorial, $0$-bisimple, $0$-$E$-unitary inverse semigroup (see \cite[Section~9.3]{Lawson-1998}).

A {\it topological} ({\it inverse}) {\it semigroup} is a Hausdorff
topological space together with a continuous semigroup operation
(and an~inversion, respectively). Obviously, the inversion defined
on a topological inverse semigroup is a homeomorphism. If $S$ is
a~semigroup (an~inverse semigroup) and $\tau$ is a topology on $S$
such that $(S,\tau)$ is a topological (inverse) semigroup, then we
shall call $\tau$ a (\emph{inverse}) \emph{semigroup}
\emph{topology} on $S$.  A {\it semitopological semigroup} is a
Hausdorff topological space endowed with a separately continuous
semigroup operation.

Let $\mathfrak{STSG}_0$ be a class of topological semigroups. A
semigroup $S\in\mathfrak{STSG}_0$ is called {\it $H$-closed in}
$\mathfrak{STSG}_0$, if $S$ is a closed subsemigroup of any
topological semigroup $T\in\mathfrak{STSG}_0$ which contains $S$
both as a subsemigroup and as a topological space. The $H$-closed
topological semigroups were introduced by Stepp in
\cite{Stepp-1969}, and there they were called {\it maximal
semigroups}. A topological semigroup $S\in\mathfrak{STSG}_0$ is
called {\it absolutely $H$-closed in the class} $\mathfrak{STSG}_0$,
if any continuous homomorphic image of $S$ into
$T\in\mathfrak{STSG}_0$ is $H$-closed in $\mathfrak{STSG}_0$.
Absolutely $H$-closed topological semigroups were introduced by
Stepp in~\cite{Stepp-1975}, and there they were called {\it
absolutely maximal}.

Recall \cite{Aleksandrov-1942}, a topological group $G$ is called {\it absolutely
closed} if $G$ is a closed subgroup of any topological group which contains $G$ as a subgroup. In our terminology such topological groups are called $H$-closed in the class of topological groups. In \cite{Raikov-1946} Raikov proved that a topological group $G$ is absolutely closed if and only if it is Raikov complete, i.e., $G$
is complete with respect to the two-sided uniformity. A topological group $G$ is called {\it $h$-complete} if for every continuous homomorphism $h\colon G\to H$ the subgroup $f(G)$ of $H$ is closed~\cite{Dikranjan-Uspenskij-1998}. In our terminology such topological groups are called absolutely $H$-closed in the class of topological groups. The $h$-completeness is preserved under taking products and closed central subgroups~\cite{Dikranjan-Uspenskij-1998}. $H$-closed paratopological and topological groups in the class of paratopological groups were studied in \cite{Ravsky-2003}. The paper \cite{Bardyla-Gutik-Ravsky-2016-?} contains a sufficient condition for a quasitopological group to be $H$-closed, which allowed us to solve a problem by Arhangel'skii and Choban \cite{Arhangelskii-Choban-2009} and show that a topological group $G$ is $H$-closed in the class of quasitopological groups if and only if $G$ is Ra\v{\i}kov-complete. In \cite{Gutik-2014} it is proved that a topological group $G$ is $H$-closed in the class of semitopological inverse semigroups with continuous inversion if and only if $G$ is compact.

In \cite{Stepp-1975} Stepp studied $H$-closed topological
semilattices in the class of topological semigroups. There he proved
that an algebraic semilattice $E$ is algebraically $h$-complete in
the class of topological semilattices if and only if every chain in
$E$ is finite. In \cite{Gutik-Repovs-2008} Gutik and Repov\v{s}
studied the closure of a linearly ordered topological
semilattice in a topological semilattice. They obtained a characterization of $H$-closed linearly ordered topological semilattices in
the class of topological semilattices and showed that every
$H$-closed linear topological semilattice is absolutely $H$-closed in the
class of topological semilattices. Such semilattices were studied also
in \cite{Chuchman-Gutik-2007, Gutik-Pagon-Repovs-2010}. In
\cite{Bardyla-Gutik-2012} the closures of the discrete
semilattices $(\mathbb{N},\min)$ and $(\mathbb{N},\max)$ were
described. In that paper the authors constructed an example of an $H$-closed
topological semilattice in the class of topological semilattices,
which is not absolutely $H$-closed in the class of topological
semilattices. The constructed example gives a negative answer to
Question~17 from \cite{Stepp-1975}. $H$-closed and absolutely
$H$-closed (semi)topological semigroups and their extensions in
different classes of topological and semitopological semigroups were
studied in \cite{Berezovski-Gutik-Pavlyk-2010, Gutik-2014,
Gutik-Lawson-Repovs-2009, Gutik-Pavlyk-2001, Gutik-Pavlyk-2003,
Gutik-Pavlyk-2005, Gutik-Pavlyk-Reiter-2009, Gutik-Reiter-2009,
Gutik-Reiter-2010}

In \cite{Bardyla-Gutik-2016-?} we showed that the $\lambda$-polycyclic monoid for in infinite cardinal $\lambda\geqslant 2$ has similar algebraic properties to that of the polycyclic monoid $P_n$ with finitely many $n\geqslant 2$ generators. In particular we proved that for every infinite cardinal $\lambda$ the polycyclic monoid $P_{\lambda}$ is a congruence-free, combinatorial, $0$-bisimple, $0$-$E$-unitary, inverse semigroup. Also we showed that every non-zero element $x\in P_\lambda$ is an isolated point in $(P_{\lambda},\tau)$ for every Hausdorff topology on $P_{\lambda}$, such that $P_{\lambda}$ is a semitopological semigroup; moreover, every locally compact Hausdorff semigroup topology on $P_\lambda$ is discrete. The last statement extends results of the paper \cite{Mesyan-Mitchell-Morayne-Peresse-20??} treating topological inverse graph semigroups. We described all feebly compact topologies $\tau$ on $P_{\lambda}$ such that $\left(P_{\lambda},\tau\right)$ is a semitopological semigroup.
 Also in \cite{Bardyla-Gutik-2016-?} we proved that for every cardinal $\lambda\geqslant 2$ any continuous homomorphism from a topological semigroup $P_\lambda$ into an arbitrary countably compact topological semigroup is annihilating and there exists no  Hausdorff feebly compact topological semigroup containing $P_{\lambda}$ as a dense subsemigroup.

This paper is a continuation of \cite{Bardyla-Gutik-2016-?}. In this paper we give sufficient conditions on a topological inverse $\lambda$-polycyclic monoid $P_{\lambda}$ to be absolutely $H$-closed in the class of topological inverse semigroups. For every infinite cardinal $\lambda$ we construct the coarsest semigroup inverse topology $\tau_{mi}$ on $P_\lambda$ and give an example of a topological inverse monoid S which contains the polycyclic
monoid $P_2$ as a dense discrete subsemigroup.

\medskip

It is well known that for an arbitrary topological inverse semigroup $S$ and every element $x\in S$ the continuity of the semigroup operation and the inversion in $S$ implies that any $\mathscr{L}$-class $L_x$ and any $\mathscr{R}$-class $R_x$ which contain the element $x$ are closed subsets in $S$. Indeed, the Wagner--Preston Theorem (see Theorem~1.17 from \cite{Clifford-Preston-1961-1967}) implies that $L_x=L_{x^{-1}x}$ and $R_x=R_{xx^{-1}}$ for arbitrary $x\in S$ and since the maps $\varphi\colon S\rightarrow E(S)$ and $\psi\colon S\rightarrow E(S)$ defined by the formulae
\begin{equation*}
    (x)\varphi=xx^{-1} \qquad \hbox{~and~} \qquad (x)\psi=x^{-1}x
\end{equation*}
are continuous, we get that $L_x=(x^{-1}x)\psi^{-1}$ and
$R_x=(xx^{-1})\varphi^{-1}$ are closed subsets of the topological
semi\-group $S$. This implies that for any idempotents $e$ and $f$
of a topological inverse semigroup $S$ the following
$\mathscr{H}$-classes of $S$:
\begin{equation*}
    H_e=R_e\cap L_e \qquad \hbox{~and~} \qquad  H_{e,f}=R_e\cap L_f
\end{equation*}
are closed subsets of the topological inverse semi\-group $S$ too.
Moreover, the relations $\mathscr{L}$, $\mathscr{R}$ and
$\mathscr{H}$ are closed subsets in $S\times S$, but $\mathscr{D}$
and $\mathscr{J}$ are not necessary closed subsets in $S\times S$
for an arbitrary topological inverse semigroup $S$ (see
\cite[Section~II]{Eberhart-Selden-1969}).

The following proposition describes $\mathscr{D}$-equivalent $\mathscr{H}$-classes in an arbitrary topological inverse semigroup.

\begin{proposition}\label{proposition-5.1}
Let $S$ be a Hausdorff topological inverse semigroup and $a,c$ be $\mathscr{D}$-equivalent elements of $S$. Then there exists $b\in S$ such that $a\mathscr{R}b$ and $b\mathscr{L}c$ in $S$, and hence $as=b$, $bs^{\prime}=a$, $tb=c$, $t^{\prime}c=b$, for some $s,s^{\prime},t,t^{\prime}\in S$. The mappings $\mathfrak{f}_{a,c}\colon H_a\rightarrow H_c\colon x\mapsto txs$ and $\mathfrak{f}_{c,a}\colon H_c\rightarrow H_a\colon x\mapsto t^{\prime}xs^{\prime}$ are continuous and mutually inverse, and hence are homeomorphisms of closed subspaces $H_a$ and $H_c$ of the topological space $S$. Moreover, if $H_a$ and $H_c$ are subgroups of $S$ then $H_a$ and $H_c$ are topologically isomorphic closed topological subgroups in the topological inverse semigroup $S$.
\end{proposition}

\begin{proof}
The above arguments imply that $H_a$ and $H_c$ are closed subspaces of $S$. Also, the algebraic part of the statement of our theorem follows from Theorem~2.3 of \cite{Clifford-Preston-1961-1967} and Theorem~1.2.7 from \cite{Higgins-1992}. The continuity of the semigroup operation in $S$ implies that the maps $\mathfrak{f}_{a,c}\colon H_a\rightarrow H_c$ and $\mathfrak{f}_{c,a}\colon H_c\rightarrow H_a$ are continuous and hence are homeomorphisms. Now, the proof of Theorem~1.2.7 from \cite{Higgins-1992} implies that in the case when $H_a$ and $H_c$ are subgroups of $S$, then there exist $u,u^{\prime}\in S$ such that the maps $\mathfrak{f}_{a,c}\colon H_a\rightarrow H_c\colon x\mapsto uxu^{\prime}$ and $\mathfrak{f}_{c,a}\colon H_c\rightarrow H_a\colon x\mapsto u^{\prime}xu$ are mutually inverse isomorphisms and the continuity of the semigroup operation in $S$ implies that so defined maps are topological isomorphisms.
\end{proof}

\begin{remark}\label{remark-5.2}
The proof of Proposition~\ref{proposition-5.1} implies that any two $\mathscr{D}$-equivalent $\mathscr{H}$-classes of a Hausdorff semitopological semigroup $S$ are homeomorphic subspaces in $S$, but they are not necessary closed subspaces in $S$, and a similar statement holds for maximal subgroups in $S$ (see \cite{Gutik-2014}).
\end{remark}

\begin{lemma}\label{lemma-5.3}
Let $T$ and $S$ be a Hausdorff topological inverse semigroup such that $S$ is an inverse subsemigroup of $T$. Let $G$ be an $\mathscr{H}$-class in $S$ which is a closed subset of the topological inverse semigroup $T$ and $D_G$ be a $\mathscr{D}$-class of the semigroup $S$ which contains the set $G$. Then every $\mathscr{H}$-class $H\subseteq D_G$ of the semigroup $S$ is a closed subset of the topological space $T$.
\end{lemma}

\begin{proof}
First we consider the case when $G$ has an idempotent, i.e., $G$ is a maximal subgroup of the semigroup $S$ (see Theorem~2.16 of \cite{Clifford-Preston-1961-1967}).

In the case when the $\mathscr{H}$-class $H$ contains an idempotent, Theorem~2.16 in  \cite{Clifford-Preston-1961-1967} implies that $H$ is a maximal subgroup of $S$ and hence $H$ is a subgroup of topological inverse semigroup $T$. We put $e$ and $f$ are unit elements of the groups $G$ and $H$, respectively. Since the idempotents $e$ and $f$ are $\mathscr{D}$-equivalent in $S$, Proposition~3.2.5 of~\cite{Lawson-1998} implies that there exists $a\in S$ such that $aa^{-1}=e$ and $a^{-1}a=f$. Now by Proposition~3.2.11(5) of~\cite{Lawson-1998} the idempotents $e$ and $f$ are $\mathscr{D}$-equivalent in the semigroup $T$. Put $H_e^T$ and $H_f^T$ be the $\mathscr{H}$-classes of idempotents $e$ and $f$ in the semigroup $T$, respectively. We define the maps $\mathfrak{f}_{e,f}\colon T\rightarrow T$ and $\mathfrak{f}_{f,e}\colon T\rightarrow T$ by the formulae $(x)\mathfrak{f}_{e,f}=a^{-1}xa$ and $(x)\mathfrak{f}_{f,e}=axa^{-1}$, respectively. Then for any $s\in H_e^T$ and $t\in H_f^T$ we get the equalities
\begin{equation*}
\begin{split}
& (s)\mathfrak{f}_{e,f}\big((s)\mathfrak{f}_{e,f}\big)^{-1}=a^{-1}sa (a^{-1}sa)^{-1}= a^{-1}saa^{-1}s^{-1}a=a^{-1}ses^{-1}a=a^{-1}ss^{-1}a=a^{-1}ea=a^{-1}a=f, \\
& \big((s)\mathfrak{f}_{e,f}\big)^{-1}(s)\mathfrak{f}_{e,f}=(a^{-1}sa)^{-1}a^{-1}sa= a^{-1}s^{-1}aa^{-1}sa= a^{-1}s^{-1}esa=a^{-1}s^{-1}sa=a^{-1}ea=a^{-1}a=f, \\
& (t)\mathfrak{f}_{f,e}\big((t)\mathfrak{f}_{f,e}\big)^{-1}=ata^{-1}(ata^{-1})^{-1}= ata^{-1}at^{-1}a^{-1}=atft^{-1}a^{-1}=att^{-1}a^{-1}=afa^{-1}=aa^{-1}=e, \\
&\big((t)\mathfrak{f}_{f,e}\big)^{-1}(t)\mathfrak{f}_{f,e}=(ata^{-1})^{-1}ata^{-1}= at^{-1}a^{-1}ata^{-1}= at^{-1}fta^{-1}=at^{-1}ta^{-1}=afa^{-1}=aa^{-1}=e,\\
& \big((s)\mathfrak{f}_{e,f}\big)\mathfrak{f}_{f,e}=aa^{-1}saa^{-1}=ese=s,\\
& \big((t)\mathfrak{f}_{f,e}\big)\mathfrak{f}_{e,f}=a^{-1}ata^{-1}a=ftf=t,
\end{split}
\end{equation*}
because $aa^{-1}=ss^{-1}=s^{-1}s=e$, $ea=a$, $af=a$ and $a^{-1}a=tt^{-1}=t^{-1}=f$. Similarly, for arbitrary $s,v\in H_e^T$ and $t,u\in H_f^T$ we have that
\begin{equation*}
    (s)\mathfrak{f}_{e,f}(v)\mathfrak{f}_{e,f}=a^{-1}saa^{-1}va=a^{-1}seva=a^{-1}sva=(sv)\mathfrak{f}_{e,f}
\end{equation*}
and
\begin{equation*}
    (t)\mathfrak{f}_{f,e}(u)\mathfrak{f}_{f,e}=ata^{-1}aua^{-1}=atfua^{-1}=atua^{-1}=(tu)\mathfrak{f}_{f,e}.
\end{equation*}
\vskip3pt

\noindent Hence the restrictions $\mathfrak{f}_{e,f}|_{H_e^T}\colon H_e^T\rightarrow H_f^T$ and $\mathfrak{f}_{f,e}|_{H_f^T}\colon H_f^T\rightarrow H_e^T$ are mutually invertible group isomorphisms. Also, since $a\in S$ we get that the restrictions $\mathfrak{f}_{e,f}|_{G}\colon G\rightarrow H$ and $\mathfrak{f}_{f,e}|_{H}\colon H\rightarrow G$ are mutually invertible group isomorphisms too. This and the continuity of left and right translations in $T$ imply that $H$ is a closed subgroup of the topological inverse semigroup $T$.

Next we consider the case when the $\mathscr{H}$-class $H$ contains no idempotents. Then there exists distinct idempotents $e,f\in S$ such that $ss^{-1}=e$ and $s^{-1}s=f$ for all $s\in H$. Suppose to the contrary that $H$ is not a closed subset of the topological inverse semigroup $T$. Then there exists an accumulation point $x\in T\setminus H$ of the set $H$ in the topological space $T$. Since every $\mathscr{H}$-class of a topological inverse semigroup $T$ is a closed subset of $T$ we get that $H$ and $x$ are contained in a same $\mathscr{H}$-class $H_x$ of the semigroup $T$. Then $xx^{-1}=e$ and $x^{-1}x=f$. Now the $\mathscr{H}$-class $H_e^T$ in $T$ which contains the idempotent $e\in S$ is a topological subgroup of the topological inverse semigroup $T$ and by Proposition~\ref{proposition-5.1} the subspace $H_e^T$ of the topological space $T$ is homeomorphic to the subspace $H_x$ of $T$. Moreover, Theorem~1.2.7 from \cite{Higgins-1992} implies that there exists a homeomorphism $\mathfrak{f}\colon H_x\rightarrow H_e^T$ such that the image $(H)\mathfrak{f}$ is a topological subgroup of the topological inverse semigroup $T$ and $(H)\mathfrak{f}$ is topologically isomorphic to the topological group $G$. Then $(H)\mathfrak{f}$ is not a closed subgroup of $T$ which contradicts our above part of the proof.

Assume that $G$ has no idempotents. By the previous part of the proof it suffices to show that there exists a maximal subgroup $H_e$ with an idempotent $e$ in the $\mathscr{D}$-class $D_G$ such that $H_e$ is a closed subgroup of topological semigroup $T$. Suppose to the contrary that every maximal subgroup in the $\mathscr{D}$-class $D_G$ is not a closed in $T$. Fix and arbitrary subgroup $H_e$ with an idempotent $e$ in the $\mathscr{D}$-class $D_G$ such that $xx^{-1}=e$ for all $x\in G$. Then Proposition~3.2.11(3) of~\cite{Lawson-1998} implies that there exist $\mathscr{H}$-classes $H_G^T$ and $H_e^T$ in the semigroup $T$ which contain the set $G$ and group $H_e$. Since in the topological semigroup $T$ every $\mathscr{H}$-class is a closed subset in $T$, we have that $G$ is a closed subset of the space $H_G^T$ and $H_e$ is not a closed subgroup of the topological group $H_e^T$. Then Proposition~3.2.11  of~\cite{Lawson-1998} and Proposition~\ref{proposition-5.1} imply that there exist $s,s^{\prime},t,t^{\prime}\in S$  such that the maps $\mathfrak{f}_{e}\colon H_e^T\rightarrow H_G^T\colon x\mapsto txs$ and $\mathfrak{f}_{G}\colon H_G^T\rightarrow H_e^T\colon x\mapsto t^{\prime}xs^{\prime}$ are mutually invertible homeomorphisms of the topological spaces $H_e^T$ and $H_G^T$ such that the restrictions $\mathfrak{f}_{e}|_{H_e}\colon H_e^T\rightarrow G$ and $\mathfrak{f}_{G}|_{G}\colon G\rightarrow H_e$ are mutually invertible homeomorphisms. This is a contradiction, because $H_e$ is not a closed subset of $H_e^T$. This completes the proof of the lemma.
\end{proof}

Lemma~\ref{lemma-5.3} implies the following corollary.

\begin{corollary}\label{corollary-5.4}
Let $T$ and $S$ be a Hausdorff topological inverse semigroup such
that $S$ is an inverse subsemigroup of $T$. Let $G$ be a maximal
subgroup in $S$ which is $H$-closed in the class of topological
inverse semigroups and $D_G$ be a $\mathscr{D}$-class of the
semigroup $S$ which contains the group $G$. Then every
$\mathscr{H}$-class $H\subseteq D_G$ of the semigroup $S$ is a
closed subset of the topological space $T$.
\end{corollary}

\begin{lemma}\label{lemma-5.5}
Let $S$ be a Hausdorff topological inverse semigroup such the following conditions hold:
\begin{itemize}
  \item[$(i)$] every maximal subgroup of the semigroup $S$ is $H$-closed in the class topological groups;
  \item[$(ii)$] all non-minimal elements of the semilattice $E(S)$ are isolated points in $E(S)$.
\end{itemize}
If there exists a topological inverse semigroup $T$ such that $S$ is
a dense subsemigroup of $T$ and $T\setminus S\neq\varnothing$ then
for every $x\in T\setminus S$ at least one of the points $x\cdot
x^{-1}$ or $x^{-1}\cdot x$ belongs to $T\setminus S$.
\end{lemma}

\begin{proof}
First we consider the case when the topological semilattice $E(S)$ does not have the smallest element. Then the space $E(S)$ is discrete and Theorem~3.3.9 of~\cite{Engelking-1989} implies that $E(S)$ is an open subset of the topological space $E(T)$ and hence every point of the set $E(S)$ is isolated in $E(T)$. Also by Proposition~II.3 \cite{Eberhart-Selden-1969} we have that $\operatorname{cl}_{T}(E(S))=\operatorname{cl}_{E(T)}(E(S))$ and hence the points of the set $E(T)\setminus E(S)$ are not isolated in the space $E(T)$.

Fix an arbitrary point $x\in T\setminus S$. By Corollary~\ref{corollary-5.4} every $\mathscr{H}$-class is a closed subset of the topological inverse semigroup $T$. Since $x$ is an accumulation point of the set $S$ in the topological space $T$ we have that every open neighbourhood $U(x)$ of the point $x$ in $T$ intersects infinitely many $\mathscr{H}$-classes of the semigroup $S$. By Proposition~II.1 of~\cite{Eberhart-Selden-1969} the inversion on $T$ is a homeomorphism of the topological space $T$ and hence $\left(U(x)\right)^{-1}$ is an open neighbourhood of the point $x^{-1}$ in $T$ which intersects infinitely many $\mathscr{H}$-classes of the semigroup $S$. Then the continuity of the semigroup operations and the inversion in $T$ implies that at least one of the sets $\left(U(x)\left(U(x)\right)^{-1}\right)\cap E(T)$ or $\left(\left(U(x)\right)^{-1}U(x)\right)\cap E(T)$ is infinite for every open neighbourhood $U(x)$ of the point $x$ in the topological semigroup $T$. This implies that at least one of $x\cdot x^{-1}$ or $x^{-1}\cdot x$ is a non-isolated point in the topological space $E(T)$.

In the case when the semilattice $E(S)$ has a minimal idempotent the
presented above arguments imply that for arbitrary point $x\in
T\setminus S$ and every open neighbourhood $U(x)$ of the point $x$
in $T$ one of the sets $\left(U(x)\left(U(x)\right)^{-1}\right)\cap E(T)$
or $\left(\left(U(x)\right)^{-1}U(x)\right)\cap E(T)$ is infinite
for every open neighbourhood $U(x)$ of the point $x$ in the
topological semigroup $T$. Since $H_{e}$ is a minimal ideal of $S$ and
it is a Ra\u\i kov complete topological group. Then there exists an open
neighborhood $U(x)$ of $x$ in $T$, such that $U(x)\cap
H_{e}=\emptyset$. If $xx^{-1}=e$ or $x^{-1}x=e$ then $x=xx^{-1}x\in
H_{e}$, which contradicts that $x\in T\setminus S$. Hence $xx^{-1}\in T\setminus S$ or
$x^{-1}x\in T\setminus S$.
\end{proof}

Lemma~\ref{lemma-5.5} implies the following two corollaries.

\begin{corollary}\label{corollary-5.6}
Let $S$ be a Hausdorff topological inverse semigroup satisfying the following conditions:
\begin{itemize}
  \item[$(i)$] every maximal subgroup of the semigroup $S$ and the semilattice $E(S)$ are $H$-closed in the class of topological inverse semigroups;
  \item[$(ii)$] all non-minimal elements of the semilattice $E(S)$ are isolated points in $E(S)$.
\end{itemize}
Then $S$ is  $H$-closed in the class of topological inverse semigroups.
\end{corollary}

\begin{corollary}\label{corollary-5.7}
Let $\lambda\geqslant 2$ and let $P_{\lambda}$ be a proper dense subsemigroup of a topological inverse semigroup $S$. Then either $xx^{-1}\in S\setminus P_{\lambda}$ or $x^{-1}x \in S\setminus P_{\lambda}$ for every $x\in S\setminus P_{\lambda}$.
\end{corollary}

The following theorem gives sufficient condition when a topological inverse $\lambda$-polycyclic monoid $P_{\lambda}$ is absolutely $H$-closed in the class of topological inverse semigroups.

\begin{theorem}\label{theorem-5.8}
Let $\lambda$ be a cardinal $\geqslant 2$ and  $\tau$ be a Hausdorff
inverse semigroup topology on $P_{\lambda}$ such that $U(0)\cap L$
is an infinite set for every open neighborhood $U(0)$ of zero $0$ in
$(P_{\lambda},\tau)$ and every maximal chain $L$ of the semilattice
$E(P_{\lambda})$. Then $(P_{\lambda},\tau)$ is absolutely $H$-closed
in the class of topological inverse semigroups.
\end{theorem}

\begin{proof}
First we observe that the definition of the $\lambda$-polycyclic monoid $P_{\lambda}$ implies that for every maximal chain $L$ in $E(P_{\lambda})$ the set $L\setminus\{0\}$ is an $\omega$-chain. Then Theorem~2 of~\cite{Bardyla-Gutik-2012} implies that every maximal chain $L$ in $E(P_{\lambda})$ with the induced topology from $(P_{\lambda},\tau)$ is an absolutely $H$-closed topological semilattice. Suppose that $E(P_{\lambda})$ with the induced topology from $(P_{\lambda},\tau)$ is not an $H$-closed topological semilattice. Then there exists a topological semilattice $S$ which contains $E(P_{\lambda})$ as a dense proper subsemilattice. Also the continuity of the semilattice operation in $S$ implies that zero $0$ of $E(P_{\lambda})$ is zero in $S$. Fix an arbitrary element $x\in S\setminus E(P_{\lambda})$. Then for an arbitrary open neighbourhood $U(x)$ of the point $x$ in $S$ such that $0\notin U(x)$ the continuity of the semilattice operation in $S$ implies that there exists an open neighbourhood $V(x)\subseteq U(x)$ of $x$ in $S$ such that $V(x)\cdot V(x)\subseteq U(x)$. Now, the neighbourhood $V(x)$ intersects infinitely many maximal chains of the semilattice $E(P_{\lambda})$, because all maximal chains in $E(P_{\lambda})$ with the induced topology from $(P_{\lambda},\tau)$ are absolutely $H$-closed topological semilattices. Then the semigroup operation of $P_{\lambda}$ implies that $0\in V(x)\cdot V(x)\subseteq U(x)$, which contradicts the choice of the neighbourhood $U(0)$. Therefore, $E(P_{\lambda})$ with the induced topology from $(P_{\lambda},\tau)$ is an $H$-closed topological semilattice.

Now, by Corollary \ref{corollary-5.6} the topological inverse semigroup $(P_{\lambda},\tau)$ is $H$-closed in the class of topological inverse semigroups. Since the $\lambda$-polycyclic monoid $P_{\lambda}$ is congruence free, every continuous homomorphic image of $(P_{\lambda},\tau)$ is $H$-closed in the class of topological inverse semigroups. Indeed, if $h\colon (P_{\lambda},\tau)\rightarrow T$ is a continuous (algebraic) homomorphism from $(P_{\lambda},\tau)$ into a topological inverse semigroup $T$, then the set $U(h(0))\cap h(L)$ is infinite for every open neighbourhood $U(h(0))$ of the image zero $h(0)$ in $T$. Then the previous part of the proof implies that $h(P_{\lambda})$ is a closed subsemigroup of $T$.
\end{proof}

\begin{remark}\label{remark-5.9}
By Remark~2.6 from \cite{Jones-Lawson-2014} (also see \cite[p.~453]{Jones-Lawson-2014}, \cite[Section~2.1]{Jones-2011} and \cite[Proposition~9.3.1]{Lawson-1998}) for every positive integer $n\geqslant 2$ any non-zero element $x$ of the polycyclic monoid $P_n$ has the form $u^{-1}v$, where $u$ and $v$ are elements of the free monoid $\mathscr{M}_n$, and the semigroup operation on $P_n$ in this representation is defined in the following way:
\begin{equation}\label{eq-5.1}
    a^{-1}b\cdot c^{-1}d=
    \left\{
      \begin{array}{ccl}
        (c_{1}a)^{-1}d, & \hbox{if~~} c=c_{1}b & \hbox{for some~} c_1\in \mathscr{M}_n;\\
        a^{-1}b_{1}d,   & \hbox{if~~} b=b_{1}c & \hbox{for some~} b_1\in \mathscr{M}_n;\\
        0,              & \hbox{otherwise}   &
      \end{array}
    \right.
    \quad \hbox{~and~} \quad a^{-1}b\cdot 0=0\cdot a^{-1}b=0\cdot 0=0.
\end{equation}
Then Lemma~2.4 of \cite{Bardyla-Gutik-2016-?} implies that every any non-zero element $x$ of the polycyclic monoid $P_{\lambda}$ has the form $u^{-1}v$, where $u$ and $v$ are elements of the free monoid $\mathscr{M}_{\lambda}$, and the semigroup operation on $P_{\lambda}$ in this representation is defined by formula (\ref{eq-5.1}).
\end{remark}

Now we shall construct a topology $\tau_{\textsf{mi}}$ on the $\lambda$-polycyclic monoid $P_\lambda$ such that $(P_{\lambda},\tau_{\textsf{mi}})$ is absolutely $H$-closed in the class of topological inverse semigroups.

\begin{example}\label{example-5.11}
We define a topology $\tau_{\textsf{mi}}$ on the polycyclic monoid
$P_{\lambda}$ in the following way. All non-zero elements of
$P_{\lambda}$ are isolated point in
$\left(P_{\lambda},\tau_{\textsf{mi}}\right)$. For an arbitrary
finite subset $A$ of $\mathscr{M}_{\lambda}$ put
\begin{equation*}
U_{A}(0)=\{a^{-1}b\colon a,b\in \mathscr{M}_{\lambda}\setminus\{A\}\}.
\end{equation*}
We put $\mathscr{B}_{\textsf{mi}}=\{U_{A}(0)\colon A \hbox{~is a
finite subset of~} \mathscr{M}_{\lambda}\}$ to be a base of the
topology $\tau_{\textsf{mi}}$ at zero $0\in P_{\lambda}$.

We observe that $\tau_{\textsf{mi}}$ is a Hausdorff topology on
$P_{\lambda}$ because $U_{\{a,b\}}(0)\not\ni a^{-1}b$ for every
non-zero element $a^{-1}b\in P_{\lambda}$. Also, since
$\left(U_{A}(0)\right)^{-1}=U_{A}(0)$ for any
$U_{A}(0)\in\mathscr{B}_{\textsf{mi}}$, the inversion is continuous
in $(P_\lambda,\tau_{\textsf{mi}})$. Fix an arbitrary $a^{-1}b\in
P_{\lambda}$ and any basic neighbourhood $U_{A}(0)$ of zero in
$(P_\lambda,\tau_{\textsf{mi}})$. Let $S_b$ be a set of all suffixes
of the word $b$. Put $B=P_b\cup\{kb\in\mathscr{M}_{\lambda}\colon
ka\in A\}$. It is obvious that the set $B$ is finite and hence
formula (\ref{eq-5.1}) implies that $a^{-1}b\cdot U_{B}(0)\subseteq
U_{A}(0)$. Let $S_a$ be a set of all suffixes of the word $a$. Put
$D=S_a\cup\{ta\in\mathscr{M}_{\lambda}\colon tb\in A\}$. It is
obvious that the set $D$ is finite and hence formula (\ref{eq-5.1})
implies that $U_{D}(0)\cdot a^{-1}b\subseteq U_{A}(0)$. Also
$U_{T}(0)\cdot U_{T}(0)\subseteq U_{A}(0)$ for $T=A\cup\{b\in
\mathscr{M}_{\lambda}\colon \hbox{~b is a suffix of some~} a\in
A\}$. Therefore  $(P_{\lambda},\tau_{\textsf{mi}})$ is a topological
inverse semigroup.
\end{example}

Theorem~\ref{theorem-5.8} and Example~\ref{example-5.11} implies the following corollary.

\begin{corollary}\label{corollary-5.12}
The topological inverse semigroup $(P_{\lambda},\tau_{\mathsf{mi}})$ is absolutely $H$-closed in the class of topological inverse semigroups.
\end{corollary}

\begin{definition}[\cite{Gutik-Pavlyk-2005}]\label{def-5.13} A~Hausdorff topological (inverse) semigroup
$(S, \tau)$ is said to be {\it minimal} if no Hausdorff semigroup
(inverse) topology on $S$ is strictly contained in $\tau$. If $(S,
\tau)$ is minimal topological (inverse) semigroup, then $\tau$ is
called a {\it minimal (inverse) semigroup topology.}
\end{definition}

Minimal topological groups were introduced independently in the early 1970's by Do\"{\i}tchinov~\cite{Doitchinov-1972} and
Stephenson~\cite{Stephenson-1971}. Both authors were motivated by the theory of minimal topological spaces, which was well understood at that time
(cf.~\cite{Berri-Porter-Stephenson-1971}). More than 20 years earlier L.~Nachbin~\cite{Nachbin-1949} had studied minimality in the context of division rings, and
B.~Banaschewski~\cite{Banaschewski-1974} investigated minimality in the more general setting of topological algebras. In \cite{Gutik-Pavlyk-2005} on the infinite semigroup of $\lambda\times\lambda$-matrix units $B_\lambda$ the minimal semigroup and the minimal semigroup inverse topologies were constructed.

\begin{theorem}\label{theorem-5.14} For any infinite cardinal $\lambda$,   $\tau_{\mathsf{mi}}$ is the coarsest inverse semigroup
topology on $P_\lambda$, and hence is minimal inverse semigroup.
\end{theorem}

\begin{proof}
The definition of the topology $\tau_{\mathsf{mi}}$ on $P_\lambda$ implies that the subsemigroup of idempotents $E(P_\lambda)$ of the semigroup $P_\lambda$ is a compact subset of the space $(P_{\lambda},\tau_{\mathsf{mi}})$. By Proposition 3.1 of \cite{Bardyla-Gutik-2016-?} every non zero-element of a semitopological monoid $(P_{\lambda},\tau)$ is an isolated point in the space $(P_{\lambda},\tau)$. This and above arguments imply that the topology $\tau_{\mathsf{mi}}$ on $P_\lambda$ induces the coarsest semigroup topology on the semilattice $E(P_\lambda)$. Also by Remark~2.6 from \cite{Jones-Lawson-2014} (also see \cite[p.~453]{Jones-Lawson-2014}, \cite[Section~2.1]{Jones-2011} and \cite[Proposition~9.3.1]{Lawson-1998})we have that every non-zero element of the semilattice $E(P_\lambda)$ can be represented in the form $a^{-1}a$ where $a$ are elements of the free monoid $\mathscr{M}_n$, and the semigroup operation on $E(P_\lambda)$ in this representation is defined by formula (\ref{eq-5.1}).

Also, we observe that for any topological inverse semigroup $S$ the following maps $\varphi\colon S\to E(S)$ and $\psi\colon S\to E(S)$ defines by the formulae $\varphi(x)=xx^{-1}$ and $\psi(x)=x^{-1}x$, respectively, are continuous. Since the inverse element of $u^{-1}v$ in $P_\lambda$ is equal to $v^{-1}u$, we have that $U_A=P_\lambda\setminus\left(\varphi^{-1}(A)\cup\psi^{-1}(A)\right)$, for any finite subset $A$ of the free monoid $\mathscr{M}_n$. This implies that $U_A(A)\in\tau$ for every inverse semigroup topology $\tau$ on $P_\lambda$, where $A$ is an arbitrary finite subset of $\mathscr{M}_n$. Thus, $\tau_{\mathsf{mi}}$ is the coarsest inverse semigroup
topology on the $\lambda$-polycyclic monoid $P_\lambda$.
\end{proof}

In the next example we construct a topological inverse monoid $S$ which contains the polycyclic monoid $P_2=\left\langle p_1,p_2\mid p_1p_1^{-1}=p_2p_2^{-1}=1, p_1p_2^{-1}=p_2p_1^{-1}=0\right\rangle$ as a dense discrete subsemigroup, i.e., the polycyclic monoid $P_2$ with the discrete topology is not $H$-closed in the class of topological inverse semigroups. Also, later we assume that the free monoid $\mathscr{M}_2$ is generated by two element $p_1$ and $p_2$.

\begin{example}\label{example-5.15}
Let $\mathscr{F}$ be the filter on the bicyclic semigroup $\mathscr{C}(p_1,p_1^{-1})=\langle p_1,p_1^{-1}\mid p_1p_1^{-1}=1\rangle$, generated by the base $\mathscr{B}=\left\{U_{n}\colon n\in \mathbb{N}\right\}$, where $U_{n}=\left\{p_{1}^{-k}p_{1}^{m}\colon k,m> n\right\}$. We denote
\begin{equation*}
A=\left\{a^{-1}b\in P_{2}\colon a\neq p_{1}a_{1} \hbox{~and~} b\neq p_{1}b_{1} \hbox{~for any~} a_{1},b_{1}\in \mathscr{M}_{2}\right\}.
\end{equation*}

For any element $a^{-1}b$ of the set $A$ let  $\mathscr{F}_{a^{-1}b}$ be the filter on $P_2$, generated by the base $\mathscr{B}_{a^{-1}b}=\left\{V_{n}\colon n\in \mathbb{N}\right\}$, where $V_{n}=a^{-1}U_{n}b= \left\{(p_{1}^{k}a)^{-1}p_{1}^{m}b\colon k,m> n\right\}$. It is obvious that $\mathscr{F}=\mathscr{F}_{1^{-1}1}$, where $1$ is the unit element of the free monoid $\mathscr{M}_{2}$.

We extend the binary operation from $P_{2}$ onto $S= P_{2}\cup \left\{\mathscr{F}_{a^{-1}b}\colon a^{-1}b\in A\right\}$ by the following formulae:
\begin{itemize}
  \item[\textsf{(I)}] $a^{-1}b\cdot\mathscr{F}_{c^{-1}d}=
    \left\{
      \begin{array}{cl}
        \mathscr{F}_{(ea)^{-1}d}, & \hbox{if~}  c= eb ;\\
        \mathscr{F}_{(e)^{-1}d}, & \hbox{if~} b= p_{1}^{n}c  \hbox{ for some~} n\in\mathbb{N}, \hbox{~where~} e \hbox{ is the longest suffix of~} a \\
                              &\hbox{ such that } e\neq p_{1}f \hbox{ for some~} f\in M_{2};\\
        0, & \hbox{otherwise};
      \end{array}
    \right.$

  \item[\textsf{(II)}] $\mathscr{F}_{c^{-1}d}\cdot a^{-1}b=
    \left\{
      \begin{array}{cl}
        \mathscr{F}_{c^{-1}eb}, & \hbox{if~}  d= ea ;\\
        \mathscr{F}_{c^{-1}e}, & \hbox{if~} a= p_{1}^{n}d  \hbox{ for some~} n\in\mathbb{N}, \hbox{ where~} e \hbox{ is the longest suffix of~} b \\
                              & \hbox{ such that } e\neq p_{1}f \hbox{ for some~} f\in M_{2} ;\\
        0, & \hbox{otherwise};
      \end{array}
    \right.$

  \item[\textsf{(III)}] $\mathscr{F}_{a^{-1}b}\cdot \mathscr{F}_{c^{-1}d}=
    \left\{
      \begin{array}{cl}
       \mathscr{F}_{a^{-1}d}, & \hbox{if~} b=c;\\
        0, & \hbox{otherwise}.
      \end{array}
    \right.$
\end{itemize}

It is obvious that the subset $T=S\setminus P_{2}\cup \{0\}$ with the induced binary operation from $S$ is isomorphic to the semigroup of $\omega\times\omega$-matrix units $B_\omega$ and moreover we have that $(\mathscr{F}_{a^{-1}b})^{-1}= \mathscr{F}_{b^{-1}a}$ in $T$.

We determine a topology $\tau$ on the set $S$ in the following way: assume that the elements of the semigroup $P_{2}$ are isolated points in $(S,\tau)$ and the family
\begin{equation*}
\mathscr{B}({\mathscr{F}_{a^{-1}b}})=\left\{U_{n}(\mathscr{F}_{a^{-1}b}):  U_{n}\in\mathscr{B}_{a^{-1}b}\right\}
\end{equation*}
of the set $U_{n}(\mathscr{F}_{a^{-1}b})=U_{n}\cup \{\mathscr{F}_{a^{-1}b}\}$
is a neighborhood base of the topology $\tau$ at the point $\mathscr{F}_{a^{-1}b}\in S$.

Now we show that so defined binary operation on $(S,\tau)$ is continuous.

In case \textsf{(I)} we consider three cases.

If $a^{-1}b\cdot\mathscr{F}_{c^{-1}d}=0$ then we have that  $a^{-1}b\cdot U_{n}(\mathscr{F}_{c^{-1}d})=\{0\}$ for any positive integer $n$.

If $a^{-1}b\cdot\mathscr{F}_{c^{-1}d}= \mathscr{F}_{(ea)^{-1}d}$ then $c=eb$. We claim that $a^{-1}b\cdot U_{n}(\mathscr{F}_{c^{-1}d})\subseteq U_{n}(\mathscr{F}_{(ea)^{-1}d})$ for any open basic neighbourhood $U_{n}(\mathscr{F}_{(ea)^{-1}d})$ of the point $\mathscr{F}_{(ea)^{-1}d}$ in $(S,\tau)$. Indeed, if $x\in U_{n}(\mathscr{F}_{c^{-1}d})$ then $x=(p_{1}^{m}c)^{-1}p_{1}^{k}d$ for some positive integers $m,k> n$, and hence we have that
\begin{equation*}
a^{-1}b\cdot (p_{1}^{m}c)^{-1}p_{1}^{k}d= a^{-1}b\cdot (p_{1}^{m}eb)^{-1}p_{1}^{k}d= (p_{1}^{m}ea)^{-1}p_{1}^{k}d \in U_{n}(\mathscr{F}_{(ea)^{-1}d}).
\end{equation*}

If $a^{-1}b\cdot\mathscr{F}_{c^{-1}d}= \mathscr{F}_{e^{-1}d}$, then  $e$ is the longest suffix of the word $a$ in $\mathscr{M}_2$ which is not equal to the word $p_{1}f$ for some $f\in \mathscr{M}_2$. This holds when $b=p_{1}^{t}c$ for some positive integer $t$. We claim that $a^{-1}b\cdot U_{n+t}(\mathscr{F}_{c^{-1}d})\subseteq U_{n}(\mathscr{F}_{e^{-1}d})$ for any open basic neighbourhood $U_{n}(\mathscr{F}_{e^{-1}d})$ of the point $\mathscr{F}_{e^{-1}d}$ in $(S,\tau)$.
Indeed, if $x\in U_{n+t}(\mathscr{F}_{c^{-1}d})$, then $x=(p_{1}^{m+t}c)^{-1}p_{1}^{k+t}d$ for some positive integers $m,k> n$, and hence we have that
\begin{equation*}
a^{-1}b\cdot (p_{1}^{m+t}c)^{-1}p_{1}^{k+t}d= e^{-1}p_{1}^{-l}p_{1}^{t}c\cdot (p_{1}^{m+t}c)^{-1}p_{1}^{k+t}d= (p_{1}^{m+l}e)^{-1}p_{1}^{k+t}d \in U_{n}(\mathscr{F}_{e^{-1}d}).
\end{equation*}

In case \textsf{(II)} the proof of the continuity of binary operation in $(S,\tau)$ is similar to case \textsf{(I)}.

Now we consider case \textsf{(III)}.

If $\mathscr{F}_{a^{-1}b}\cdot \mathscr{F}_{c^{-1}d}=0$ then $U_{n}(\mathscr{F}_{a^{-1}b})\cdot U_{n}(\mathscr{F}_{c^{-1}d})\subseteq \{0\}$, for any open basic neighbourhoods $U_{n}(\mathscr{F}_{a^{-1}b})$ and $U_{n}(\mathscr{F}_{c^{-1}d})$ of the points $\mathscr{F}_{a^{-1}b}$ and $\mathscr{F}_{c^{-1}d}$ in $(S,\tau)$, respectively.

If $\mathscr{F}_{a^{-1}b}\cdot \mathscr{F}_{c^{-1}d}=\mathscr{F}_{a^{-1}d}$ then $b=c$ and for every any open basic neighbourhood $U_{n}(\mathscr{F}_{a^{-1}d})$ of the point $\mathscr{F}_{a^{-1}d}$ in $(S,\tau)$ we have that $U_{n}(\mathscr{F}_{a^{-1}b})\cdot U_{n}(\mathscr{F}_{b^{-1}d})\subseteq U_{n}(\mathscr{F}_{a^{-1}d})$. Indeed if $(p_{1}^{k}a)^{-1}p_{1}^{t}b\in U_{n}(\mathscr{F}_{a^{-1}b})$ and  $(p_{1}^{l}b)^{-1}p_{1}^{m}d\in U_{n}(\mathscr{F}_{b^{-1}d})$ then
\begin{equation*}
(p_{1}^{k}a)^{-1}p_{1}^{t}b\cdot (p_{1}^{l}b)^{-1}p_{1}^{m}d= (p_{1}^{k}a)^{-1}p_{1}^{t}(b\cdot b^{-1})p_{1}^{-l}p_{1}^{m}d= (p_{1}^{s}a)^{-1}p_{1}^{z}d,
\end{equation*}
for some positive integers $s,z> n$, and hence $(p_{1}^{s}a)^{-1}p_{1}^{z}d\in U_{n}(\mathscr{F}_{a^{-1}d})$.

Thus, we proved that the binary operation on $(S,\tau)$ is continuous. Taking into account that $P_2$ is a dense subsemigroup of $(S,\tau)$ we conclude that $(S,\tau)$ is a topological semigroup. Also, since $T=S\setminus P_{2}\cup \{0\}$ with the induced binary operation from $S$ is isomorphic to the semigroup of $\omega\times\omega$-matrix units $B_\omega$ we have that idempotents in $S$ commute and moreover $\mathscr{F}_{a^{-1}b}\cdot\mathscr{F}_{b^{-1}a}\cdot\mathscr{F}_{a^{-1}b}= \mathscr{F}_{b^{-1}a}$. This implies that $S$ is an inverse semigroup. Also, for every open basic neighbourhood $U_{n}(\mathscr{F}_{a^{-1}b})$ of the point $\mathscr{F}_{a^{-1}b}$ in $(S,\tau)$ we have that $\left(U_{n}(\mathscr{F}_{a^{-1}b})\right)^{-1}=U_{n}(\mathscr{F}_{b^{-1}a})$ for all $n\in\mathbb N$ and hence the inversion in $(S,\tau)$ is continuous.
\end{example}


\section*{Acknowledgements}

We acknowledge Taras Banakh for his comments and suggestions.




\end{document}